\documentclass[11pt,reqno]{amsart}

\usepackage[T1]{fontenc}
\usepackage{lmodern}

\usepackage[a4paper,left=2.7cm,right=2.7cm,top=2.4cm,bottom=2.4cm]{geometry}
\usepackage{amsmath,amssymb,amsfonts,mathtools,mathrsfs,bm}
\usepackage{enumitem}
\usepackage[dvipsnames]{xcolor}
\usepackage{hyperref}
\usepackage{bookmark}

\hypersetup{
        colorlinks=true,
        linkcolor=MidnightBlue,
        citecolor=RoyalBlue,
        urlcolor=RoyalBlue,
        pdfauthor={Xingyu Wang},
        pdftitle={The sharp-interface limit of the matrix-valued Allen--Cahn equation}
}

\allowdisplaybreaks[3]
\numberwithin{equation}{section}
\newtheorem{theorem}{Theorem}[section]
\newtheorem{proposition}[theorem]{Proposition}

\newtheorem{lemma}[theorem]{Lemma}

\theoremstyle{definition}
\newtheorem{definition}[theorem]{Definition}
\newtheorem{remark}[theorem]{Remark}

\setlist[enumerate]{leftmargin=2.5em,itemsep=0.25em,topsep=0.35em}

\newcommand{\mat}[1]{\mathbf{#1}}
\renewcommand{\aa}{\mat{A}}
\newcommand{\aae}{\aa_\ee}
\newcommand{\aaek}{\aa_{\ee_k}}
\newcommand{\bb}{\mat{B}}
\newcommand{\cc}{\mat{C}}
\newcommand{\ii}{\mat{I}}

\newcommand{\qq}{\mat{Q}}
\newcommand{\rr}{\mat{R}}

\newcommand{\xx}{\mat{X}}
\newcommand{\yy}{\mat{Y}}
\newcommand{\zz}{\mat{Z}}

\newcommand{\br}{\mathbb{R}}
\newcommand{\rd}{\mathbb{R}^d}

\newcommand{\rk}{\mathbb{R}^k}
\newcommand{\Mn}{\mathbb{M}_n}
\newcommand{\An}{\mathbb{A}_n}
\newcommand{\ee}{\varepsilon}
\renewcommand{\o}{\Omega}

\newcommand{\on}{\mathbb{O}_n}
\newcommand{\onp}{\mathbb{O}_n^+}
\newcommand{\onn}{\mathbb{O}_n^-}
\newcommand{\onpm}{\mathbb{O}_n^\pm}
\newcommand{\hn}{\mathcal{H}^{d-1}}
\newcommand{\dd}{\mathrm{d}}
\newcommand{\dx}{\,\dd x}
\newcommand{\dt}{\,\dd t}

\newcommand{\dgt}{\mathrm{d}_{\Gamma_t}}

\newcommand{\wf}{\widetilde f}
\newcommand{\wff}{\widetilde F}
\newcommand{\dwf}{\mathrm d_{\wff}}
\newcommand{\dwfe}{\mathrm d_{\wff}^{\ee}}
\newcommand{\pe}{\psi_\ee}
\newcommand{\hhe}{H_\ee}
\newcommand{\nne}{n_\ee}

\newcommand{\p}{\partial}
\newcommand{\sumod}{\sum_{i=1}^d}

\DeclarePairedDelimiter{\norm}{\lVert}{\rVert}
\DeclarePairedDelimiter{\abs}{\lvert}{\rvert}
\newcommand{\loc}{\mathrm{loc}}

\DeclareMathOperator{\dist}{dist}
\DeclareMathOperator{\tr}{tr}

\DeclareMathOperator{\spt}{spt}

\renewcommand{\div}{\operatorname{div}}

\title{The sharp-interface limit of the matrix-valued Allen--Cahn equation}
\author{Xingyu Wang}
\email{xingyuwang@sjtu.edu.cn}
\begin{document}

\begin{abstract}
	We study the sharp-interface limit of a matrix-valued Allen--Cahn equation with the Saint Venant--Kirchhoff potential
	\[
		F(\aa)=\frac14\norm{\aa\aa^\top-\ii}^2 .
	\]
	The zero set of this potential is the orthogonal group \(\on=\onp\cup\onn\), and the corresponding limiting problem combines mean-curvature motion of the interface with harmonic-map heat flow in the two bulk phases.  The proof combines a modulated-energy argument with compactness estimates
obtained from two skew-symmetric commutator formulations of the equation. The method avoids the spectral analysis of linearized operators around quasi-minimal connecting orbits and the construction of high-order matched asymptotic expansions.  In particular, the limiting maps satisfy the minimal-pair condition on the moving interface and the weak transmission identities which, for smooth limits, are equivalent to the Neumann-type jump condition of the sharp-interface system.
\end{abstract}

\maketitle
\setcounter{tocdepth}{1}
\tableofcontents

\section{Introduction}
\subsection{Background}

The sharp-interface limit in variational dynamics is a classical topic in
partial differential equations and geometric analysis.  The scalar
Allen--Cahn equation
\[
    \p_t u_\ee=\Delta u_\ee-\ee^{-2}F'(u_\ee),
    \qquad u_\ee:\Omega\times(0,T)\to\br,
\]
with the standard double-well potential
\(F(u)=\frac14(1-u^2)^2\), is the prototypical model. As
\(\ee\downarrow0\), solutions approach a two-phase state, and the separating
interface evolves by mean curvature flow.  The corresponding level-set and
viscosity-solution formulations of mean-curvature motion were developed in
\cite{Chen1991,Evans1991}, and the convergence of Allen--Cahn type equations to
mean-curvature motion was studied in
\cite{Evans1992,Ilmanen1993,Soner1997}.

Keller, Rubinstein, and Sternberg extended this picture to vector-valued
order parameters \cite{Rubinstein1989fast,Rubinstein1989reaction}.  In their
setting the zero set of the potential consists of two disjoint smooth manifolds
\(N^\pm\subset\rk\).  Formal asymptotic expansions predict a coupled limiting
system: the interface moves by mean curvature, while the two bulk phases evolve
by harmonic-map heat flow into \(N^\pm\).  The rigorous justification of this
prediction is the Keller--Rubinstein--Sternberg problem.  Related elliptic and
harmonic-map aspects of phase transitions with higher-dimensional potential
wells were studied in \cite{Lin2012,Lin2019}.  A main difficulty is
that, when the wells are non-trivial manifolds, the interfacial energy may
depend on the limiting traces of the bulk maps.  In the fully minimally paired case this dependence is absent, whereas in
the partially minimally paired case it produces a genuine coupling between
the interface and the bulk dynamics \cite{Fei2023,Lin2012}.  The notions of
fully and partially minimally paired potentials are recalled in
Definition \ref{def:minimal-connection}.

Matrix-valued Allen--Cahn type models and their interface dynamics have also
appeared in numerical and asymptotic studies; see, for instance,
\cite{Wang2019}.  Fei, Lin, Wang, and Zhang \cite{Fei2023} systematically
studied the dynamic sharp-interface limit of the matrix-valued Allen--Cahn
equation
\begin{equation}\label{allen-cahn1}
    \p_t\aae=\Delta\aae-\ee^{-2}(\aae\aae^\top\aae-\aae),
    \qquad
    \aae:\Omega\times(0,T)\to\Mn .
\end{equation}
Here \(\Mn\) is the space of real \(n\times n\) matrices and
\[
    D F(\aa)=\aa\aa^\top\aa-\aa,
    \qquad
    F(\aa)=\frac14\norm{\aa\aa^\top-\ii}^2 .
\]
The potential vanishes precisely on the orthogonal group
\[
    \on=\onp\cup\onn,
\]
where \(\onpm\) denotes the set of orthogonal matrices with determinant
\(\pm1\).  The limiting sharp-interface system consists of harmonic-map heat
flow in the two bulk phases, mean-curvature motion of the interface, the
minimal-pair condition, and a Neumann-type transmission condition:
\begin{subequations}\label{system}
    \begin{align}
        \p_t\aa_\pm
        &=
        \Delta\aa_\pm
        -
        \sumod \p_i\aa_\pm\aa_\pm^\top\p_i\aa_\pm,
        && \aa_\pm\in\onpm
        \quad\text{in }\Omega_t^\pm,
        \label{limit-heat-flow}
        \\
        V&=\kappa,
        &&\text{on }\Gamma_t,
        \label{limit-mcf}
        \\
        (\aa_+,\aa_-)&\text{ is a minimal pair},
        &&\text{on }\Gamma_t,
        \label{limit-dirichlet}
        \\
        \p_\nu\aa_+&=\p_\nu\aa_-,
        &&\text{on }\Gamma_t.
        \label{limit-neumann}
    \end{align}
\end{subequations}
Here \(V\), \(\kappa\), and \(\nu\) denote the normal velocity, the mean
curvature, and the chosen unit normal to \(\Gamma_t\), respectively.  For
\(n\ge2\), not every pair in \(\onp\times\onn\) is minimal.  This partial
minimal pairing prevents a direct application of the classical
matched-asymptotic construction of de Mottoni and Schatzman
\cite{DeMottoni1995}.  To overcome this obstacle, \cite{Fei2023} introduced
quasi-minimal connecting orbits and developed a delicate spectral and
orthogonal-decomposition theory for the corresponding linearized operators.

A different approach is provided by the modulated-energy, or relative-entropy,
method.  For the scalar Allen--Cahn equation, this method gives a direct
stability estimate of a diffuse interface around a given smooth
mean-curvature-flow interface \cite{Fischer2020a}; related relative-entropy
ideas also appear in \cite{Jerrard2015,Fischer2020}.  For vector-valued sharp-interface problems, related nematic--isotropic
transition models in liquid crystals were studied in
\cite{Fei2015,Fei2018,Laux2021}, and the relative-entropy method has been
applied to several fully minimally paired models \cite{Laux2021,Liu2024}.  More recently, Liu
\cite{Liu2025} justified the Keller--Rubinstein--Sternberg limit for a class of
truncated squared-distance potentials with two manifold-valued wells, including
partially minimally paired cases.  In that general framework, however, the
Neumann-type jump condition \eqref{limit-neumann} is not recovered.

The purpose of this paper is to give a modulated-energy proof of the
sharp-interface limit for the matrix-valued equation \eqref{allen-cahn1}.  The key point is that the argument not only yields compactness away from
the interface, but also recovers weak transmission identities which, for
smooth limits, are equivalent to the Neumann-type jump condition
\eqref{limit-neumann}.  This is
possible because the Saint Venant--Kirchhoff potential has an additional
algebraic structure that is absent in a general two-well potential.

The main ingredients are as follows.
\begin{itemize}
    \item \emph{Two commutator formulations.}
    Following the Chen--Shatah wedge-product idea, we rewrite
\eqref{allen-cahn1} in two skew-symmetric commutator forms.  For
\(\xx,\aa\in\Mn\), define
\[
    [\xx,\aa]_L:=\xx\aa^\top-\aa\xx^\top,
    \qquad
    [\aa,\xx]_R:=\aa^\top\xx-\xx^\top\aa .
\]
Both brackets take values in the space \(\An\) of skew-symmetric matrices.
Since
\[
    [D F(\aa),\aa]_L=0,
    \qquad
    [\aa,D F(\aa)]_R=0,
\]
smooth solutions of \eqref{allen-cahn1} satisfy
\begin{subequations}\label{commutator-equations}
    \begin{align}
        [\p_t\aae,\aae]_L
        =
        \sumod \p_i[\p_i\aae,\aae]_L,
        \label{commu-eq-}
        \qquad
        [\aae,\p_t\aae]_R
        =
        \sumod \p_i[\aae,\p_i\aae]_R.
    \end{align}
\end{subequations}
Both identities are needed.  In the limiting smooth regime, the left and right
skew-symmetric trace identities together are equivalent to
\(\p_\nu\aa_+=\p_\nu\aa_-\) under the minimal-pair constraint.
   \item  \emph{A matrix-adapted quasi-distance.} We construct a mollified quasi-distance \(\dwfe\) adapted to the
matrix-valued potential.  It satisfies the commutation laws
\[
    [D\dwfe(\aa),\aa]_L=0,
    \qquad
    [\aa,D\dwfe(\aa)]_R=0,
\]
and preserves the Modica-type differential inequality needed in the
modulated-energy argument, up to a lower-order error.  These two properties
allow the dissipative part of the modulated-energy inequality to control the
left and right commutators in \eqref{commutator-equations}.

\item \emph{Compactness and transmission.} The modulated energy gives uniform estimates away from the moving interface.  The commutator bounds then allow us to pass to the
limit in \eqref{commutator-equations}.  The resulting two weak transmission
identities are equivalent to \(\p_\nu\aa_+=\p_\nu\aa_-\) whenever the limiting
maps are smooth up to the interface.
\end{itemize}

\subsection{Main results}

We first fix some notation.  Throughout the paper,
\[
    \on:=\{\aa\in\Mn:\aa\aa^\top=\ii\},
    \qquad
    \onpm:=\{\aa\in\on:\det\aa=\pm1\},
\]
and
\[
    \An:=\{\xx\in\Mn:\xx^\top=-\xx\}.
\]

Let \(\Omega\subset\rd\) be a bounded smooth domain, and let \(\aa^{\mathrm{bd}}\in C^\infty(\overline\Omega;\onn) \) be a prescribed time-independent boundary datum. We consider the initial-boundary value problem 
\begin{subequations}\label{allen-cahn} 
\begin{align} \p_t\aae &= \Delta\aae-\ee^{-2}(\aae\aae^\top\aae-\aae), &&\text{in }\Omega\times(0,T), \label{allen-cahn-equation} \\ \aae &= \aa_{\ee,0}, &&\text{on }\Omega\times\{0\}, \label{allen-initial} \\ \aae &= \aa^{\mathrm{bd}}, &&\text{on }\p\Omega\times(0,T). \label{allen-boundary} 
\end{align} 
\end{subequations}

Let \(\{\Gamma_t\}_{t\in[0,T]}\) be a smooth family of closed hypersurfaces evolving by mean curvature flow and remaining strictly inside \(\Omega\). The hypersurface \(\Gamma_t\) separates \(\Omega\) into two open sets \(\Omega_t^\pm\), and we set 
\[ 
\Omega_T^\pm := \bigcup_{t\in(0,T)}\Omega_t^\pm\times\{t\}, \qquad \Gamma := \bigcup_{t\in[0,T]}\Gamma_t\times\{t\}. 
\] Let \(\dgt(x,t)\) be the signed distance to \(\Gamma_t\), chosen positive in \(\Omega_t^+\) and negative in \(\Omega_t^-\). For \(\delta>0\), define 
\[ 
\Gamma_t(\delta) := \{x\in\Omega:\abs{\dgt(x,t)}<\delta\}, \qquad \Gamma(\delta) := \{(x,t)\in\Omega\times[0,T]:\abs{\dgt(x,t)}<\delta\}. 
\] 
We assume that there exists \(\delta_\Gamma\in(0,1)\) such that
\[
P_{\Gamma_t}:\Gamma_t(\delta_\Gamma)\to\Gamma_t 
\]
is smooth for every \(t\in[0,T]\), and 
\[ 
\dist(\Gamma_t,\p\Omega)\ge \delta_\Gamma \qquad\text{for all }t\in[0,T]. 
\] 
In particular, \(\p\Omega\subset\Omega_t^-\) for all \(t\in[0,T]\).

\begin{theorem}[Sharp-interface limit under well-prepared data]\label{thm:main}
    Let \(\aa_{\ee,0}\in H^1(\Omega;\Mn)\cap L^\infty(\Omega)\) be initial data for \eqref{allen-cahn} satisfying \(\aa_{\ee,0}=\aa^{\mathrm{bd}}\) on \(\partial\Omega\). We assume that \(\{\aa_{\ee,0}\}\) is well prepared relative to
\(\Gamma_0\), namely, there exist \(C_0>0\) and \(    \aa_{0,\pm}\in L^2(\Omega_0^\pm;\onpm)\)
such that
\begin{equation}\label{well-prepared-assumption}
    \ee\|\aa_{\ee,0}\|_{L^\infty(\Omega)}
    +
    E_\ee[\aa_{\ee,0}\mid\Gamma_0]
    \le C_0\ee ,
\end{equation}
and
\begin{equation}\label{initial-strong-convergence}
    \aa_{\ee,0}
    \to
    \aa_{0,+}\chi_{\Omega_{0}^+}+\aa_{0,-}\chi_{\Omega_{0}^-}
    \qquad\text{strongly in }L^2(\Omega)
    \quad\text{as }\ee\downarrow0 .
\end{equation}
Here \(C_0\) is independent of \(\ee\), and
\(E_\ee[\aa_{\ee,0}\mid\Gamma_0]\) is defined in
\eqref{relative-definition}.

    Let \(\aae\) be the solution of \eqref{allen-cahn}.  Then there exist a
    sequence \(\ee_k\downarrow0\) and maps
    \[
        \aa_\pm
        \in
        H^1((0,T);L^2(\Omega_t^\pm;\onpm))
        \cap
        L^\infty((0,T);H^1(\Omega_t^\pm;\onpm)),\quad \aa_\pm(0)=\aa_{0,\pm}
    \]
    such that
    \begin{equation}\label{main-local-convergence}
        \aa_{\ee_k}
        {\rightharpoonup}
        \aa_\pm
        \quad\text{weakly in }
        H^1_{\text{loc}}(\Omega_T^\pm).
    \end{equation}
    Moreover, the following assertions hold.

    \begin{enumerate}[label=\textup{(\arabic*)},leftmargin=2.5em]
        \item For every \(\Phi\in C_c^\infty((0,T)\times\Omega;\An)\),
        \begin{equation}\label{inter-equality-left}
            \sum_{\pm}\int_0^T\int_{\Omega_t^\pm}
            \left(
            \p_t\aa_\pm\aa_\pm^\top:\Phi
            +
            \sumod \p_i\aa_\pm\aa_\pm^\top:\p_i\Phi
            \right)\dx\dt=0 .
        \end{equation}

        \item For every \(\Psi\in C_c^\infty((0,T)\times\Omega;\An)\),
        \begin{equation}\label{inter-equality-right}
            \sum_{\pm}\int_0^T\int_{\Omega_t^\pm}
            \left(
            \aa_\pm^\top\p_t\aa_\pm:\Psi
            +
            \sumod \aa_\pm^\top\p_i\aa_\pm:\p_i\Psi
            \right)\dx\dt=0 .
        \end{equation}

        \item For a.e. \(t\in(0,T)\),
        \begin{equation}\label{result-minimal}
            \norm{\aa_+(x,t)-\aa_-(x,t)}=2
            \quad\text{for }\hn\text{-a.e. }x\in\Gamma_t .
        \end{equation}
    \end{enumerate}
\end{theorem}
\begin{remark}
\indent 

Compared with the matched-asymptotic approach of \cite{Fei2023}, our proof does
not require spectral stability of linearized operators around quasi-minimal
connecting orbits.  Compared with the general modulated-energy framework of
\cite{Liu2025}, the matrix-valued structure allows us to pass to the limit in
two skew-symmetric commutator identities and thereby recover the missing
Neumann-type transmission condition.  

  As a byproduct, the theorem constructs a weak solution to the sharp-interface limiting system \eqref{system} along the prescribed smooth mean-curvature-flow interface \(\Gamma_t\).  This complements the weak-solution construction for the matrix-valued two-phase harmonic map flow obtained by minimizing-movement methods in \cite{Wang2025}.
\end{remark}

\begin{remark}\label{rem:main-comments}
	The following comments clarify the content of Theorem \ref{thm:main}.
	\begin{enumerate}[label=\textup{(\roman*)},leftmargin=2.5em]
		\item \textit{Meaning of the weak identities.} The two identities \eqref{inter-equality-left} and \eqref{inter-equality-right} are the global weak formulation of the limiting bulk equations together with the transmission laws on the moving interface.  If the test function is supported compactly in one phase, say in \(\Omega_T^+\), then \eqref{inter-equality-left} gives
		      \[
			      \p_t\aa_+\aa_+^\top
			      -\sumod\p_i(\p_i\aa_+\aa_+^\top)=0
			      \qquad\text{in }\mathcal D'(\Omega_T^+).
		      \]
		      Since \(\aa_+\aa_+^\top=\ii\), differentiating the constraint gives
		      \[
			      \p_i\aa_+\aa_+^\top+\aa_+\p_i\aa_+^\top=0 .
		      \]
		      Multiplying the preceding distributional identity on the right by \(\aa_+\) therefore yields
		      \[
			      \p_t\aa_+=\Delta\aa_+
			      -\sumod\p_i\aa_+\aa_+^\top\p_i\aa_+
			      \qquad\text{in }\Omega_T^+.
		      \]
		      The same argument applies in \(\Omega_T^-\).  Thus \eqref{inter-equality-left}, equivalently \eqref{inter-equality-right}, contains the harmonic-map heat flow into \(\onpm\) in the two bulk phases.  This is the first equation of the limiting system \eqref{system}.

		\item \textit{Recovery of the Neumann-type jump condition in the smooth case.}     
		      Suppose, only for interpretation, that \(\aa_\pm\) are smooth up to \(\Gamma_t\).  Let \(\nu\) be the normal on \(\Gamma_t\) pointing from \(\Omega_t^-\) to \(\Omega_t^+\).  Integrating \eqref{inter-equality-left} by parts in the two phases and using the bulk equations gives the left skew-symmetric trace identity
		      \[
			      \p_\nu\aa_+\aa_+^\top-\aa_+\p_\nu\aa_+^\top
			      =
			      \p_\nu\aa_-\aa_-^\top-\aa_-\p_\nu\aa_-^\top
			      \qquad\text{on }\Gamma_t .
		      \]
		      Similarly, \eqref{inter-equality-right} gives the right skew-symmetric trace identity
		      \[
			      \aa_+^\top\p_\nu\aa_+-\p_\nu\aa_+^\top\aa_+
			      =
			      \aa_-^\top\p_\nu\aa_- -\p_\nu\aa_-^\top\aa_-
			      \qquad\text{on }\Gamma_t .
		      \]
		      These two identities are both needed.  Indeed, by the minimal-pair condition and Lemma \ref{minimal-equi}, there is \(q\in\mathbb S^{n-1}\) such that
		      \[
			      \aa_+=\aa_-\rr,
			      \qquad
			      \rr:=\ii-2q q^\top,
			      \qquad
			      \rr^\top=\rr,
			      \quad
			      \rr^2=\ii .
		      \]
		      Set
		      \[
			      \yy:=\aa_-^\top\p_\nu\aa_- ,
			      \qquad
			      \zz:=\aa_-^\top\p_\nu\aa_+ .
		      \]
		      The left and right trace identities become, respectively,
		      \[
			      \zz=\yy\rr,
			      \qquad
			      \rr\zz=\yy .
		      \]
		      Hence \(\rr\yy\rr=\yy\).  Since \(\yy\) is skew-symmetric and \(\rr\) is the reflection with normal \(q\), this implies \(\rr\yy q=-\yy q\) and \(\yy q=0\), and therefore \(\yy\rr=\yy\).  Consequently \(\zz=\yy\), i.e.
		      \[
			      \aa_-^\top\p_\nu\aa_+=\aa_-^\top\p_\nu\aa_- .
		      \]
		      Multiplication by \(\aa_-\) gives \(\p_\nu\aa_+=\p_\nu\aa_-\).  Conversely, this Neumann condition immediately implies both skew-symmetric trace identities.  Therefore, in the classical regime, the weak transmission identities in the theorem are exactly the Neumann-type boundary condition in \eqref{limit-neumann}.

		\item \textit{Role of the minimal-pair condition.}      
		      The condition \eqref{result-minimal} is recovered from the smallness of the modulated energy by applying a one-dimensional lower bound on normal line segments across \(\Gamma_t\).  For this potential \(F(\aa)=\frac14\norm{\aa\aa^\top-\ii}^2\), Lemma \ref{minimal-equi} states that minimality is equivalent to \(\norm{\aa_+-\aa_-}=2\), or equivalently to the reflection representation \(\aa_+=\aa_-(\ii-2q q^\top)\).

		\item \textit{Use of the uniform \(L^\infty\)-bound.}       As in \cite{Liu2025}, the assumption \(\aa_{\ee,0}\in L^\infty\), together with the bounded boundary datum, gives a uniform \(L^\infty\)-bound for \(\aae\) by the maximum principle for the parabolic system.  This estimate is essential in two places.  First, it turns the projected estimates \eqref{project-estimate} for
		      \[
			      \nabla\aae-\Pi_{\aae}\nabla\aae,
			      \qquad
			      \p_t\aae-\Pi_{\aae}\p_t\aae
		      \]
		      into uniform bounds for the commutators, because
		      \[
			      \norm{[\xx,\aa]_L}+\norm{[\aa,\xx]_R}
			      \le 4\norm{\aa}\norm{\xx}.
		      \]
		      Second, it allows us to identify nonlinear weak limits such as
		      \[
			      [\p_i\aae,\aae]_L
			      \rightharpoonup
			      [\p_i\aa_\pm,\aa_\pm]_L
		      \]
		      from weak convergence of derivatives and strong local convergence of \(\aae\) in Proposition \ref{prop:limit-convergence}.

		\item \textit{Well-prepared initial data.}        
		      The theorem is stated under the well-preparedness assumption \eqref{well-prepared-assumption}.  This is the natural hypothesis for a relative-entropy proof, since it says that the diffuse interface initially has the same surface tension and location as \(\Gamma_0\), up to an error of order \(\mathcal{O}(\ee)\). Such well-prepared initial data can be constructed by the method of \cite[Theorem~1.2 and Section~6]{Liu2025}.

	\end{enumerate}
\end{remark}

\section{Preliminaries}

\subsection{Minimal connections}\label{minimal-conn}

We recall the notion of minimal connections from \cite[Section 3.1]{Fei2023}.  Let \(W:\rk\to\br_+\) be a smooth potential whose zero set is the disjoint union of two compact connected smooth manifolds \(N^\pm\subset\rk\) without boundary.  For \((p_+,p_-)\in N^+\times N^-\), consider
\begin{equation}\label{ode}
	\p_z^2 u=\nabla_uW(u),
	\qquad
	u(\pm\infty)=p_\pm,
	\qquad
	u:\br\to\rk .
\end{equation}

\begin{definition}\label{def:minimal-connection}
	A solution of \eqref{ode} is called a \emph{connecting orbit}, and \(p_\pm\) are called its \emph{ends}.  A connecting orbit is \emph{minimal} if it minimizes
	\begin{equation}\label{one-energy1}
		\mathbf e(u)=\int_\br\left(\frac12\abs{u'}^2+W(u)\right)\dd z
	\end{equation}
	among all curves in
	\[
		H^1_{p_\pm}(\br)=
		\{u\in H^1_{\loc}(\br):\lim_{z\to\pm\infty}u(z)=p_\pm\}.
	\]
	A pair \((p_+,p_-)\in N^+\times N^-\) is called a \emph{minimal pair} if it can be joined by a minimal connecting orbit.  The potential is \emph{fully minimally paired} if every pair is minimal and \emph{partially minimally paired} otherwise.
\end{definition}

For the potential \(F(\aa)=\frac14\norm{\aa\aa^\top-\ii}^2\), the minimal pairs and connecting orbits have an explicit characterization.

\begin{lemma}[Minimal pairs, \cite{Fei2023}]\label{minimal-equi}
	For \((\aa_+,\aa_-)\in\onp\times\onn\), the following statements are equivalent:
	\begin{enumerate}[label=\textup{(\roman*)},leftmargin=2.5em]
		\item \((\aa_+,\aa_-)\) is a minimal pair;
		\item \(\norm{\aa_+-\aa_-}=\min_{(\aa,\bb)\in\onp\times\onn}\norm{\aa-\bb}\);
		\item \(\norm{\aa_+-\aa_-}=2\);
		\item \(\aa_+=\aa_-(\ii-2q q^\top)\) for some \(q\in\mathbb S^{n-1}\).
	\end{enumerate}
\end{lemma}

\begin{lemma}[Minimal connecting orbits, \cite{Fei2023}]\label{minimal-orbits}
	All minimal connecting orbits of \(F\) are given by
	\begin{equation}\label{minimal-profile}
		\Theta_\tau(\aa_+,\aa_-;z)
		=s_\tau(z)\aa_+ +(1-s_\tau(z))\aa_-,
		\qquad
		s_\tau(z)=s(z+\tau),
	\end{equation}
	where \((\aa_+,\aa_-)\) is a minimal pair and
	\begin{equation}\label{logistic-profile}
		s(z)=1-(1+e^{\sqrt2 z})^{-1}.
	\end{equation}
	In particular, \(\Theta_\tau\) solves \eqref{ode} and attains the minimum in \eqref{one-energy1}.
\end{lemma}

\subsection{The quasi-distance function}\label{subsec:quasi-distance}

For \(\aa\in\Mn\), define the distances to the two connected components of
\(\on\) by
\begin{equation}\label{rho-def}
    \rho^\pm(\aa)
    :=
    \min_{\bb\in\onpm}\norm{\aa-\bb},
    \qquad
    \rho(\aa)
    :=
    \min\{\rho^+(\aa),\rho^-(\aa)\}.
\end{equation}
The following geometric lower bound is crucial.

\begin{lemma}[{\cite[Lemma~A.1]{Fei2023}}]\label{quasi-profile}
For every \(\aa\in\Mn\),
\begin{equation}\label{F-lower-bound}
    F(\aa)\ge \frac14(2-\rho(\aa))^2\rho(\aa)^2 .
\end{equation}
Equality holds if and only if
\[
    \aa=\bb(\ii-\rho(\aa)q q^\top)
\]
for some \(\bb\in\on\) and \(q\in\mathbb S^{n-1}\).  Moreover,
\(F(\aa)>1/4\) whenever \(\rho(\aa)>1\).
\end{lemma}

Motivated by Lemma \ref{quasi-profile}, define the one-dimensional comparison
potential
\begin{equation}\label{ftilde-def}
    \wf(r):=
    \begin{cases}
        \dfrac14 r^2(2-r)^2, & 0\le r\le 1,\\[0.5em]
        \dfrac14,           & r\ge 1.
    \end{cases}
\end{equation}
Set
\begin{equation}\label{Ftilde-def}
    \wff(\aa):=\wf(\rho(\aa)).
\end{equation}
Then Lemma \ref{quasi-profile} implies
\begin{equation}\label{Ftilde-le-F}
    \wff(\aa)\le F(\aa),
    \qquad
    \aa\in\Mn .
\end{equation}

The corresponding surface tension is
\begin{equation}\label{surface-tension}
    \sigma
    :=
    2\int_0^1\sqrt{2\wf(r)}\,\dd r
    =
    \frac{2\sqrt2}{3}.
\end{equation}
We define the quasi-distance function \(\dwf:\Mn\to[0,\sigma]\) by
\begin{equation}\label{dwf-def}
    \dwf(\aa):=
    \begin{cases}
        \displaystyle\int_0^{\rho^-(\aa)} \sqrt{2\wf(s)}\,\dd s,
        & \rho^-(\aa)\le 1,\\[1em]
        \dfrac{\sigma}{2},
        & \rho(\aa)\ge 1,\\[1em]
        \displaystyle
        \sigma-\int_0^{\rho^+(\aa)} \sqrt{2\wf(s)}\,\dd s,
        & \rho^+(\aa)\le 1.
    \end{cases}
\end{equation}
The three definitions agree on their common boundaries, since
\[
    \int_0^1\sqrt{2\wf(s)}\,\dd s=\frac{\sigma}{2}.
\]

\begin{lemma}\label{lemma-quasi-distance-function}
	The function \(\dwf\) is globally Lipschitz and satisfies
	\begin{equation}\label{old-differ-ineq}
		\norm{D\dwf(\aa)}\le\sqrt{2\wff(\aa)}
		\quad\text{for a.e. }\aa\in\Mn .
	\end{equation}
\end{lemma}

\begin{proof}
	The only point requiring comment is the behavior across the transition region where \(\rho(\aa)\ge1\).  For instance, if \(\rho^-(\aa)\le1\) and \(\rho(\bb)\ge1\), then
	\[
		\begin{aligned}
			\abs{\dwf(\bb)-\dwf(\aa)}
			 & =\left|\frac{\sigma}{2}
			-\int_0^{\rho^-(\aa)}\sqrt{2\wf(r)}\,\dd r\right| \\
			 & \le \int_{\rho^-(\aa)}^1\sqrt{2\wf(r)}\,\dd r
			\le \max_{0\le r\le1}\sqrt{2\wf(r)}\,\norm{\aa-\bb}.
		\end{aligned}
	\]
	The other cases are analogous, and hence \(\dwf\) is globally Lipschitz.  The differential inequality follows from the chain rule for Lipschitz functions, the identity \(\norm{D\rho}=1\) a.e., and \eqref{Ftilde-le-F}.
\end{proof}

\subsection{The mollified quasi-distance function}\label{subsec:mollified-quasi-distance}

The function \(\dwf\) introduced in \eqref{dwf-def} is globally Lipschitz, but it is not smooth across the cut locus of the two wells and across the set where the two components are equidistant.  In the modulated-energy argument we need to apply the chain rule to
\[
	\pe=\dwfe(\aae),
\]
and we also need a commutator identity compatible with the Chen--Shatah reformulation.  This is the reason for introducing a carefully chosen mollification of \(\dwf\).

A point that is important for the matrix-valued problem is that \(\dwf\) distinguishes the two components \(\onn\) and \(\onp\).  Therefore \(\dwf\) is not invariant under the full left-right action of \(\on\times\on\), since multiplication by an orthogonal matrix of negative determinant exchanges \(\onp\) and \(\onn\).   However, full \(\on\times\on\)-invariance is not needed.  To derive the
commutation laws used below, it is enough that the regularized quasi-distance be
invariant under the proper left-right action of \(\onp\times\onp\).

\begin{lemma}[Proper-orthogonal invariance of \(\dwf\)]\label{lemma:proper-invariance-dwf}
	For every \(\qq_1,\qq_2\in\onp\) and every \(\aa\in\Mn\),
	\begin{equation}\label{proper-invariance-dwf}
		\dwf(\qq_1\aa\qq_2^\top)=\dwf(\aa),
		\qquad
		\wff(\qq_1\aa\qq_2^\top)=\wff(\aa).
	\end{equation}
\end{lemma}

\begin{proof}
	Because the Frobenius norm is invariant under left and right orthogonal multiplication,
	\[
		\dist(\qq_1\aa\qq_2^\top,\onpm)
		=\dist(\aa,\qq_1^\top\onpm\qq_2).
	\]
	Since \(\det\qq_1=\det\qq_2=1\), the set \(\qq_1^\top\onpm\qq_2\) is exactly \(\onpm\).  Hence \(\rho^\pm(\qq_1\aa\qq_2^\top)=\rho^\pm(\aa)\).  The definitions \eqref{dwf-def} and \eqref{Ftilde-def} then give \eqref{proper-invariance-dwf}.
\end{proof}

Choose a non-negative radial mollifier \(\phi\in C_c^\infty(\Mn)\) satisfying
\begin{equation}\label{mollifier-normalization}
	\phi(\bb)=\widehat\phi(\norm{\bb}),
	\qquad
	\spt\phi\subset B_1^{\Mn},
	\qquad
	\int_{\Mn}\phi(\bb)\,\dd\bb=1 .
\end{equation}
In particular, \(\phi\) is invariant under the full left-right orthogonal action, and therefore also under \(\onp\times\onp\).  Choose \(K>3\) sufficiently large; for definiteness we take
\begin{equation}\label{K-choice}
    K=5 .
\end{equation}
This ensures that the mollification error and the additional constant
\(\ee^{K-1}\) in \(F_\ee\) are of higher order than the main
\(\mathcal O(\ee)\) modulated-energy error. Define
\begin{equation}\label{mollifier-def}
	\phi_\ee(\bb):=\ee^{-n^2K}\phi(\ee^{-K}\bb),
	\qquad
	\spt\phi_\ee\subset B_{\ee^K}^{\Mn},
	\qquad
	\int_{\Mn}\phi_\ee\,\dd\bb=1,
\end{equation}
and set
\begin{equation}\label{mollified-quasi}
	\dwfe(\aa):=(\phi_\ee*\dwf)(\aa)
	=\int_{\Mn} \phi_\ee(\bb)\dwf(\aa-\bb)\,\dd\bb .
\end{equation}
The following elementary estimates will be used repeatedly.

\begin{lemma}[Basic properties of the mollification]\label{lemma:mollification-basic}
	There exists a constant \(C>0\), independent of \(\ee\), such that
	\begin{align}
		0\le \dwfe\le \sigma,
		\qquad
		\norm{\dwfe-\dwf}_{L^\infty(\Mn)}  \le C\ee^K,\label{mollification-error-Linfty} 	
	\end{align}
    \begin{align}
        \norm{D\dwfe}_{L^\infty(\Mn)}      \le C.\label{mollification-gradient-bound}
    \end{align}
	Moreover, for every \(\qq_1,\qq_2\in\onp\),
	\begin{equation}\label{proper-invariance-dwfe}
		\dwfe(\qq_1\aa\qq_2^\top)=\dwfe(\aa),
		\qquad \aa\in\Mn .
	\end{equation}
\end{lemma}

\begin{proof}
	The bounds \(0\le \dwfe\le \sigma\) follow from the same bounds for \(\dwf\).  Since \(\dwf\) is globally Lipschitz and \(\phi_\ee\) is supported in \(B_{\ee^K}^{\Mn}\),
	\[
		\begin{aligned}
			\abs{\dwfe(\aa)-\dwf(\aa)}
			 & \le \int_{\Mn}\phi_\ee(\bb)
			\abs{\dwf(\aa-\bb)-\dwf(\aa)}\,\dd\bb \\
			 & \le \operatorname{Lip}(\dwf)
			\int_{\Mn}\phi_\ee(\bb)\norm{\bb}\,\dd\bb
			\le C\ee^K .
		\end{aligned}
	\]
	The gradient bound follows either by differentiating the convolution in the sense of distributions:
	\[
		\norm{D\dwfe}_{L^\infty}\le \operatorname{Lip}(\dwf).
	\]
	Finally, for \(\qq_1,\qq_2\in\onp\), using the change of variables \(\cc=\qq_1^\top\bb\qq_2\), the invariance of Lebesgue measure on \(\Mn\), the radial invariance of \(\phi_\ee\), and Lemma \ref{lemma:proper-invariance-dwf}, we get
	\[
		\begin{aligned}
			\dwfe(\qq_1\aa\qq_2^\top)
			 & =\int_{\Mn}\phi_\ee(\bb)\dwf(\qq_1\aa\qq_2^\top-\bb)\,\dd\bb \\
			 & =\int_{\Mn}\phi_\ee(\qq_1\cc\qq_2^\top)
			\dwf\bigl(\qq_1(\aa-\cc)\qq_2^\top\bigr)\,\dd\cc                \\
			 & =\int_{\Mn}\phi_\ee(\cc)\dwf(\aa-\cc)\,\dd\cc
			=\dwfe(\aa).
		\end{aligned}
	\]
\end{proof}

\begin{lemma}[Commutation laws]\label{commutator-law}
	The mollified quasi-distance \(\dwfe\) is smooth and satisfies, for every \(\aa\in\Mn\),
	\begin{equation}\label{commutator-identity}
		[D\dwfe(\aa),\aa]_L=0,
		\qquad
		[\aa,D\dwfe(\aa)]_R=0 .
	\end{equation}
\end{lemma}

\begin{proof}
	Let \(\xx\in\An\).  Since \(e^{t\xx}\in\onp\), the left invariance in \eqref{proper-invariance-dwfe} gives
	\[
		\dwfe(e^{t\xx}\aa)=\dwfe(\aa).
	\]
	Differentiating at \(t=0\),
	\[
		0=D\dwfe(\aa):(\xx\aa)
		=\tr\bigl(D\dwfe(\aa)^\top\xx\aa\bigr)
		=\tr\bigl(\aa D\dwfe(\aa)^\top\xx\bigr).
	\]
	This holds for every skew-symmetric \(\xx\).  Hence \(\aa D\dwfe(\aa)^\top\) is symmetric, i.e.
	\[
		\aa D\dwfe(\aa)^\top=D\dwfe(\aa)\aa^\top,
	\]
	which is precisely \([D\dwfe(\aa),\aa]_L=0\).

	Similarly, the right invariance gives \(\dwfe(\aa e^{t\xx})=\dwfe(\aa)\).  Differentiating at \(t=0\),
	\[
		0=D\dwfe(\aa):(\aa\xx)
		=\tr\bigl(D\dwfe(\aa)^\top\aa\xx\bigr).
	\]
	Thus \(D\dwfe(\aa)^\top\aa\) is symmetric, or equivalently
	\[
		D\dwfe(\aa)^\top\aa=\aa^\top D\dwfe(\aa),
	\]
	which is \([\aa,D\dwfe(\aa)]_R=0\).
\end{proof}

The mollification also preserves the basic differential inequality
\eqref{old-differ-ineq}, up to a small error.  

\begin{lemma}[Differential inequality]\label{lemma-differential-inequality}
	For every \(M>0\), there exists \(\ee_0=\ee_0(M)>0\) such that
	\begin{equation}\label{differential-inequality1}
		\norm{D\dwfe(\aa)}\le\sqrt{2F_\ee(\aa)},
		\qquad
		F_\ee(\aa):=F(\aa)+\ee^{K-1},
	\end{equation}
	whenever \(\norm{\aa}\le M\), \(\aa\in\Mn\), and \(0<\ee<\ee_0\).
\end{lemma}

\begin{proof}
	Because \(\dwf\) is Lipschitz, \(D\dwf\) exists a.e. and
	\[
		D\dwfe(\aa)=\int_{\Mn}\phi_\ee(\bb)D\dwf(\aa-\bb)\,\dd\bb .
	\]
	Using \eqref{old-differ-ineq} and Jensen's inequality,
	\[
		\begin{aligned}
			\norm{D\dwfe(\aa)}
			 & \le \int_{\Mn}\phi_\ee(\bb)\norm{D\dwf(\aa-\bb)}\,\dd\bb             \\
			 & \le \left(\int_{\Mn}\phi_\ee(\bb)\,2F(\aa-\bb)\,\dd\bb\right)^{1/2}.
		\end{aligned}
	\]
	If \(\norm{\aa}\le M\) and \(\norm{\bb}\le\ee^K\), then \(\aa-\bb\) remains in a fixed bounded set.  Since \(F\) is smooth, it is Lipschitz on this set, and therefore
	\[
		F(\aa-\bb)\le F(\aa)+C_M\norm{\bb}.
	\]
	Consequently,
	\[
		\norm{D\dwfe(\aa)}^2
		\le 2F(\aa)+C_M\int_{\Mn}\phi_\ee(\bb)\norm{\bb}\,\dd\bb
		\le 2F(\aa)+C_M\ee^K.
	\]
	Taking \(\ee_0(M)>0\) sufficiently small so that \(C_M\ee^K\le2\ee^{K-1}\) for \(0<\ee<\ee_0(M)\), we obtain \eqref{differential-inequality1}.
\end{proof}

\section{The modulated energy}\label{section:modulated-energy}

We extend the geometric quantities associated with \(\Gamma_t\) to a tubular neighborhood.  Let \(\varphi:\br\to[0,\infty)\) be an even smooth cut-off such that
\begin{equation}\label{psi-cutoff}
	\varphi(s)>0\text{ for }\abs{s}<1,
	\qquad
	\varphi(s)=0\text{ for }\abs{s}\ge1,
	\qquad
	1-4s^2\le\varphi(s)\le1-\frac12s^2\text{ for }\abs{s}<\frac12 .
\end{equation}
Define
\begin{equation}\label{xi-H-def}
	\xi(x,t)=\varphi\left(\frac{\dgt(x,t)}{\delta_\Gamma}\right)\nabla\dgt(x,t),
\end{equation}
and
\begin{equation}\label{H-def}
	H(x,t)=-\varphi_0(x,t)\Delta\dgt(P_{\Gamma_t}(x),t)\nabla\dgt(x,t),
\end{equation}
where \(\varphi_0\in C_c^\infty(\Gamma(\delta_\Gamma))\) and \(\varphi_0=1\) in \(\Gamma(\delta_\Gamma/2)\).  As in \cite{Fischer2020,Liu2025}, there exists \(C(\Gamma)>0\) such that
\begin{subequations}\label{eq-geometric}
	\begin{align}
		\div\xi+H\cdot\xi                                  & =\mathcal{O}(\dgt),\label{geom1} \\
		\abs{\p_t\xi+(H\cdot\nabla)\xi+(\nabla H)^\top\xi} & =\mathcal{O}(\dgt),\label{geom2} \\
		\p_t\abs{\xi}^2+(H\cdot\nabla)\abs{\xi}^2          & =\mathcal{O}(\dgt),\label{geom3} \\
		\norm{\nabla\xi}+\abs{H}+\norm{\nabla H}           & \le C(\Gamma),\label{geom4}
	\end{align}
\end{subequations}
in \(\overline\Omega\times[0,T]\), and
\begin{equation}\label{xi-H-boundary-zero}
	\xi=H=0\quad\text{on }\p\Omega\times[0,T].
\end{equation}

For the phase-field solution \(\aae\), define
\begin{equation}\label{psi-eps-def}
	\pe=\dwfe(\aae).
\end{equation}
The phase-field normal and mean-curvature vector are
\begin{equation}\label{normal-def}
	\nne=
	\begin{cases}
		\dfrac{\nabla\pe}{\abs{\nabla\pe}}, & \abs{\nabla\pe}\ne0, \\[0.7em]
		0,                                  & \abs{\nabla\pe}=0,
	\end{cases}
\end{equation}
and, componentwise for \(1\le i\le d\),
\begin{equation}\label{Heps-def}
	(\hhe)_i=
	\begin{cases}
		-\left(\ee\Delta\aae-\dfrac1\ee D F(\aae)\right):
		\dfrac{\p_i\aae}{\norm{\nabla\aae}}, & \norm{\nabla\aae}\ne0, \\[0.9em]
		0,                                   & \norm{\nabla\aae}=0.
	\end{cases}
\end{equation}
Using \eqref{allen-cahn-equation}, we have the vector identity
\begin{equation}\label{Heps-identity}
	\hhe\norm{\nabla\aae}=-\ee\,\p_t\aae:\nabla\aae .
\end{equation}

For \(0\le i\le d\), with \(\p_0=\p_t\), define the projection of \(\p_i\aae\) onto the direction of \(D\dwfe(\aae)\) by
\begin{equation}\label{projection-def}
	\Pi_{\aae}(\p_i\aae)=
	\begin{cases}
		\left(\p_i\aae:\dfrac{D\dwfe(\aae)}{\norm{D\dwfe(\aae)}}\right)
		\dfrac{D\dwfe(\aae)}{\norm{D\dwfe(\aae)}},
		   & \norm{D\dwfe(\aae)}\ne0, \\[1em]
		0, & \norm{D\dwfe(\aae)}=0.
	\end{cases}
\end{equation}
The notation \(\Pi_{\aae}\nabla\aae\) is understood componentwise. 

By Lemma \ref{commutator-law}, for \(0\le i\le d\),
\begin{subequations}\label{comm-invariance}
	\begin{align}
		[\p_i\aae,\aae]_L
		 & =[\p_i\aae-\Pi_{\aae}\p_i\aae,\aae]_L,
		\label{comm-left-invariance}              \\
		[\aae,\p_i\aae]_R
		 & =[\aae,\p_i\aae-\Pi_{\aae}\p_i\aae]_R.
		\label{comm-right-invariance}
	\end{align}
\end{subequations}
Moreover,
\begin{subequations}\label{orthogonality}
	\begin{align}
		\norm{\Pi_{\aae}\nabla\aae}\,\norm{D\dwfe(\aae)} & =\abs{\nabla\pe},\label{orthogonality-1} \\
		(\p_i\aae-\Pi_{\aae}\p_i\aae):\Pi_{\aae}\p_i\aae & =0.
		\label{orthogonality-2}
	\end{align}
\end{subequations}

The modulated energy is
\begin{equation}\label{relative-definition}
	E_\ee[\aae\mid\Gamma](t)
	=\int_\Omega
	\left(
	\frac\ee2\norm{\nabla\aae}^2
	+\frac1\ee F_\ee(\aae)
	-\xi\cdot\nabla\pe
	\right)\dx,
\end{equation}
where \(F_\ee=F+\ee^{K-1}\).  By Lemma \ref{lemma-differential-inequality} and \(\abs\xi\le1\), the integrand is non-negative after completing the square, and hence \(E_\ee[\aae\mid\Gamma](t)\ge0\).

\begin{lemma}[Coercivity of the modulated energy]\label{lemma-energy-coercivity}
	There exists \(C=C(\Gamma)>0\) such that, for every \(t\in[0,T]\),
	\begin{align}
		\int_\Omega
		\left(
		\frac\ee2\norm{\nabla\aae}^2
		+\frac1\ee F_\ee(\aae)
		+\abs{\nabla\pe}
		\right)
		\min\{\dgt^2,1\}\dx
		 & \le C E_\ee[\aae\mid\Gamma](t),
		\label{coercivity-distance}        \\
		\ee\int_\Omega\norm{\nabla\aae-\Pi_{\aae}\nabla\aae}^2\dx
		 & \le 2E_\ee[\aae\mid\Gamma](t).
		\label{coercivity-projection}
	\end{align}
\end{lemma}

\begin{proof}
	We give the details because this estimate is used repeatedly.  Put
	\[
		e_\ee(\aae):=\frac\ee2\norm{\nabla\aae}^2
		+\frac1\ee F_\ee(\aae).
	\]
	By Lemma \ref{lemma-differential-inequality},
	\[
		\abs{\nabla\pe}
		=\abs{D\dwfe(\aae):\nabla\aae}
		\le \norm{D\dwfe(\aae)}\norm{\nabla\aae}
		\le \sqrt{2F_\ee(\aae)}\norm{\nabla\aae}
		\le e_\ee(\aae).
	\]
	Consequently
	\begin{equation}\label{basic-energy-gap}
		e_\ee(\aae)-\xi\cdot\nabla\pe
		\ge e_\ee(\aae)-\abs{\xi}\abs{\nabla\pe}
		\ge (1-\abs{\xi})e_\ee(\aae),
	\end{equation}
	where we used \(\abs{\nabla\pe}\le e_\ee(\aae)\) in the last step.  The construction of \(\xi\) gives
	\begin{equation}\label{xi-defect}
		1-\abs{\xi(x,t)}\ge c(\Gamma)\min\{\dgt(x,t)^2,1\}
		\qquad\text{in }\Omega\times[0,T].
	\end{equation}
	Indeed, near \(\Gamma_t\) this follows from the condition
	\(1-\varphi(s)\ge c s^2\), and away from the tubular neighborhood \(\xi=0\).  Combining
	\eqref{basic-energy-gap} and \eqref{xi-defect} gives the part of
	\eqref{coercivity-distance} containing \(e_\ee\).  The same argument, now using
	\[
		e_\ee(\aae)-\xi\cdot\nabla\pe
		\ge e_\ee(\aae)-\abs{\xi}\abs{\nabla\pe}
		\ge (1-\abs{\xi})\abs{\nabla\pe},
	\]
	gives the corresponding bound for \(\abs{\nabla\pe}\).

	It remains to prove \eqref{coercivity-projection}.  The projection is orthogonal in
	\(\Mn\); hence, componentwise,
	\[
		\norm{\nabla\aae}^2
		=\norm{\Pi_{\aae}\nabla\aae}^2
		+\norm{\nabla\aae-\Pi_{\aae}\nabla\aae}^2 .
	\]
	Moreover,
	\[
		\abs{\nabla\pe}
		=\abs{D\dwfe(\aae):\Pi_{\aae}\nabla\aae}
		\le \norm{D\dwfe(\aae)}\norm{\Pi_{\aae}\nabla\aae}
		\le \frac\ee2\norm{\Pi_{\aae}\nabla\aae}^2
		+\frac1\ee F_\ee(\aae).
	\]
	Therefore
	\[
		\begin{aligned}
			e_\ee(\aae)-\xi\cdot\nabla\pe
			 & \ge e_\ee(\aae)-\abs{\nabla\pe}                         \\
			 & \ge \frac\ee2\norm{\nabla\aae-\Pi_{\aae}\nabla\aae}^2 .
		\end{aligned}
	\]
	After integration over \(\Omega\), this yields \eqref{coercivity-projection}.
\end{proof}

\begin{proposition}[Modulated-energy inequality]\label{prop:modulated-energy-ineq}
	There exists \(C=C(\Gamma)>0\) such that
	\begin{equation}\label{relative-entropy-inequality}
\begin{aligned}
    \frac{\dd}{\dd t}E_\ee[\aae\mid\Gamma]
    &+
    \frac1{2\ee}\int_\Omega
    \norm{\ee\p_t\aae-(\div\xi)D\dwfe(\aae)}^2\dx
    \\
    &+
    \frac1{2\ee}\int_\Omega
    \abs{\hhe-\ee\norm{\nabla\aae}H}^2\dx
    +
    \frac1{2\ee}\int_\Omega
    \left(
        \norm{\ee\p_t\aae}^2-\abs{\hhe}^2
    \right)\dx
    \\
    &\le
    C E_\ee[\aae\mid\Gamma].
\end{aligned}
\end{equation}
\end{proposition}

\begin{proof}
	This is the relative-entropy inequality of Fischer--Laux--Simon and Laux--Liu, in the form adapted to vector-valued potentials with a smooth phase function; see \cite{Fischer2020a,Laux2021,Liu2025}.  The present setting satisfies the hypotheses needed in that calculation: the vector fields \(\xi\) and \(H\) obey the transport estimates \eqref{eq-geometric}; the function \(\dwfe\) is smooth; and the differential inequality \eqref{differential-inequality1} gives the required Modica-type bound.  The added constant in \(F_\ee=F+\ee^{K-1}\) is independent of \(\aa\), so it does not change the Allen--Cahn equation and only contributes a harmless lower-order term to the energy.  Boundary terms vanish because \(\xi=H=0\) on \(\p\Omega\) and the boundary datum is time independent. 
\end{proof}

\begin{proposition}[Uniform estimates]\label{prop:uniform-estimates}
	Let \(\aae\) solve \eqref{allen-cahn} and assume
\eqref{well-prepared-assumption}. Then there exists a constant \(C\), depending only on
\(\Gamma\), \(T\), \(C_0\), and
\(\norm{\aa_{\ee,0}}_{L^\infty}\), but independent of \(\ee\), such that,
for every \(\delta\in(0,\delta_\Gamma)\),
\eqref{uniform-estimates} holds.
    \begin{subequations}\label{uniform-estimates}
\begin{align}
    \sup_{t\in[0,T]}E_\ee[\aae\mid\Gamma](t)
    &\le C\ee,
    \label{energy-small}
    \\
    \sup_{t\in[0,T]}\int_\Omega
    \norm{\nabla\aae-\Pi_{\aae}\nabla\aae}^2\dx
    +
    \int_0^T\int_\Omega
    \norm{\p_t\aae-\Pi_{\aae}\p_t\aae}^2\dx\dt
    &\le C,
    \label{project-estimate}
    \\
    \sup_{t\in[0,T]}
    \int_{\Omega_t^\pm\setminus\Gamma_t(\delta)}
    \left(
        \norm{\nabla\aae}^2+\frac{F_\ee(\aae)}{\ee^2}
    \right)\dx
    &\le C\delta^{-2},
    \label{estimate-1}
    \\
    \int_0^T\int_{\Omega_t^\pm\setminus\Gamma_t(\delta)}
    \norm{\p_t\aae}^2\dx\dt
    &\le C\delta^{-2}.
    \label{estimate-2}
\end{align}
\end{subequations}
\end{proposition}
\begin{proof}
	We start from \eqref{relative-entropy-inequality}.  The only term there which is
	not manifestly non-negative is
	\[
		\frac1{2\ee}\int_\Omega
		\left(\norm{\ee\p_t\aae}^2-\abs{\hhe}^2
		+\abs{\hhe-\ee\norm{\nabla\aae}H}^2\right)\dx .
	\]
	It is nevertheless non-negative up to the material time derivative generated by
	\(H\).  Indeed, from \eqref{Heps-identity},
	\[
		\hhe\cdot H\,\norm{\nabla\aae}
		=-\ee\p_t\aae:(H\cdot\nabla)\aae .
	\]
	Hence
	\begin{equation}\label{positive-H-block}
\begin{aligned}
    &\ee^2\norm{\p_t\aae}^2-\abs{\hhe}^2
    +\abs{\hhe-\ee\norm{\nabla\aae}H}^2
    \\
     =&
    \ee^2\norm{\p_t\aae}^2
    +\ee^2\norm{\nabla\aae}^2\abs{H}^2
    -2\ee\norm{\nabla\aae}\hhe\cdot H
    \\
     \ge&
    \ee^2\norm{\p_t\aae}^2
    +\ee^2\norm{(H\cdot\nabla)\aae}^2
    +2\ee^2\p_t\aae:(H\cdot\nabla)\aae
    \\
     =&
    \ee^2\norm{\p_t\aae+(H\cdot\nabla)\aae}^2 .
\end{aligned}
\end{equation}
	Moreover, by the orthogonality of the projection \(\Pi_{\aae}\),
	\begin{equation}\label{time-projection-square}
		\begin{aligned}
			 & \norm{\ee\p_t\aae-(\div\xi)D\dwfe(\aae)}^2       \\
			 =&\ee^2\norm{\p_t\aae-\Pi_{\aae}\p_t\aae}^2
			+\norm{\ee\Pi_{\aae}\p_t\aae-(\div\xi)D\dwfe(\aae)}^2 .
		\end{aligned}
	\end{equation}
	Dropping the non-negative square in the second line of \eqref{time-projection-square},
	using \eqref{positive-H-block}, and applying Gronwall's lemma with the initial
	smallness \eqref{well-prepared-assumption}, we obtain
	\begin{equation}\label{gronwall-inequality}
		\begin{aligned}
			\sup_{t\in[0,T]}E_\ee[\aae\mid\Gamma](t)
			 & +\ee\int_0^T\int_\Omega
			\norm{\p_t\aae-\Pi_{\aae}\p_t\aae}^2\dx\dt \\
			 & +\ee\int_0^T\int_\Omega
			\norm{\p_t\aae+(H\cdot\nabla)\aae}^2\dx\dt
			\le C\ee .
		\end{aligned}
	\end{equation}
	Together with the projection coercivity \eqref{coercivity-projection}, this gives
	\eqref{energy-small} and \eqref{project-estimate}.

	We next derive the estimates away from the interface.  If
	\(x\in\Omega_t^\pm\setminus\Gamma_t(\delta)\), then
	\(\min\{\dgt(x,t)^2,1\}\ge c\min\{\delta^2,1\}\).  Therefore
	\eqref{coercivity-distance} and \eqref{energy-small} imply
	\[
		\ee\int_{\Omega_t^\pm\setminus\Gamma_t(\delta)}
		\norm{\nabla\aae}^2\dx
		+\frac1\ee\int_{\Omega_t^\pm\setminus\Gamma_t(\delta)}
		F_\ee(\aae)\dx
		\le C\delta^{-2}\ee .
	\]
	Dividing by \(\ee\) yields
	\eqref{estimate-1}.  Finally,
	\[
		\p_t\aae
		=\bigl(\p_t\aae+(H\cdot\nabla)\aae\bigr)-(H\cdot\nabla)\aae,
	\]
	so the boundedness of \(H\), \eqref{gronwall-inequality}, and \eqref{estimate-1}
	lead to \eqref{estimate-2}.
\end{proof}

\begin{proposition}[Compactness and identification of commutators]\label{prop:limit-convergence}
	There exist limiting maps \(\aa_\pm\) with the regularity stated in Theorem \ref{thm:main} and a subsequence \(\ee_k\downarrow0\) such that
	\begin{subequations}\label{limit-convergences}
		\begin{align}
			[\p_t\aaek,\aaek]_L
			            & \rightharpoonup
			\sum_\pm[\p_t\aa_\pm,\aa_\pm]_L\chi_{\Omega_t^\pm}
			            &                                             & \text{weakly in }L^2((0,T);L^2(\Omega)),\label{limit-comm-t-left}                  \\
			[\p_i\aaek,\aaek]_L
			            & \stackrel{*}{\rightharpoonup}
			\sum_\pm[\p_i\aa_\pm,\aa_\pm]_L\chi_{\Omega_t^\pm}
			            &                                             & \text{weakly-* in }L^\infty((0,T);L^2(\Omega)),\label{limit-comm-i-left}           \\
			[\aaek,\p_t\aaek]_R
			            & \rightharpoonup
			\sum_\pm[\aa_\pm,\p_t\aa_\pm]_R\chi_{\Omega_t^\pm}
			            &                                             & \text{weakly in }L^2((0,T);L^2(\Omega)),\label{limit-comm-t-right}                 \\
			[\aaek,\p_i\aaek]_R
			            & \stackrel{*}{\rightharpoonup}
			\sum_\pm[\aa_\pm,\p_i\aa_\pm]_R\chi_{\Omega_t^\pm}
			            &                                             & \text{weakly-* in }L^\infty((0,T);L^2(\Omega)),\label{limit-comm-i-right}          \\
			\p_t\aaek   & \rightharpoonup\p_t\aa_\pm
			            &                                             & \text{weakly in }L^2_{\loc}(\Omega_T^\pm),\label{limit-t}                          \\
			\nabla\aaek & \stackrel{*}{\rightharpoonup}\nabla\aa_\pm
			            &                                             & \text{weakly-* in }L^\infty((0,T);L^2_{\loc}(\Omega_t^\pm)),\label{limit-i} \\
			\aaek       & \to\aa_\pm
			            &                                             & \text{strongly in }C([0,T];L^2_{\loc}(\Omega_t^\pm)),\label{limit-strong}          \\
			\aaek       & \to \aa:=\sum_\pm\aa_\pm\chi_{\Omega_t^\pm}
			            &                                             & \text{a.e. in }\Omega\times(0,T).
			\label{ae-convergence}
		\end{align}
	\end{subequations}
	Moreover \(\aa_\pm\) take values in \(\onpm\) a.e. in \(\Omega_T^\pm\).
\end{proposition}

\begin{proof}
	The proof is divided into several steps. 

	\smallskip
	\noindent\emph{Step 1. Global weak bounds for the commutators.}
	Let
	\[
		M:=\sup_{\ee}\norm{\aae}_{L^\infty(\Omega\times(0,T))}<\infty,
	\]
	which follows from the maximum principle and the uniform \(L^\infty\)-bound of the initial and boundary data.  By the commutator invariance identities \eqref{comm-invariance}, for \(0\le i\le d\) with \(\p_0:=\p_t\),
	\[
		[\p_i\aae,\aae]_L=[\p_i\aae-\Pi_{\aae}\p_i\aae,\aae]_L,
		\qquad
	[\aae,\p_i\aae]_R=[\aae,\p_i\aae-\Pi_{\aae}\p_i\aae]_R .
	\]
	Using the elementary estimate
	\[
		\norm{[\xx,\aa]_L}+\norm{[\aa,\xx]_R}
		\le 4\norm{\aa}\norm{\xx},
		\qquad \aa,\xx\in\Mn,
	\]
	and the projected estimate \eqref{project-estimate}, we obtain
	\begin{align*}
		[\p_t\aae,\aae]_L,
		[\aae,\p_t\aae]_R
		 & \quad\text{uniformly bounded in }L^2((0,T);L^2(\Omega)),      \\
		[\p_i\aae,\aae]_L,
		[\aae,\p_i\aae]_R
		 & \quad\text{uniformly bounded in }L^\infty((0,T);L^2(\Omega)),
		\qquad 1\le i\le d .
	\end{align*}
	Hence, after extracting a subsequence, there exist
	\[
		\cc_0^L,\cc_0^R\in L^2((0,T);L^2(\Omega;\An)),
		\qquad
		\cc_i^L,\cc_i^R\in L^\infty((0,T);L^2(\Omega;\An))
	\]
	such that
    \begin{subequations}\label{limit-comm-amb}
		\begin{align}
			[\p_t\aaek,\aaek]_L
			            & \rightharpoonup
			\cc_0^L
			            &                                             & \text{weakly in }L^2((0,T);L^2(\Omega)),\label{limit-comm-tleft}                  \\
			[\p_i\aaek,\aaek]_L
			            & \stackrel{*}{\rightharpoonup}
			\cc_i^L
			            &                                             & \text{weakly-* in }L^\infty((0,T);L^2(\Omega)),\label{limit-comm-ileft}           \\
			[\aaek,\p_t\aaek]_R
			            & \rightharpoonup
			\cc_0^R
			            &                                             & \text{weakly in }L^2((0,T);L^2(\Omega)),\label{limit-comm-tright}                 \\
			[\aaek,\p_i\aaek]_R
			            & \stackrel{*}{\rightharpoonup}
			\cc_i^R
			            &                                             & \text{weakly-* in }L^\infty((0,T);L^2(\Omega)).\label{limit-comm-iright}          
		\end{align}
	\end{subequations}
	The purpose of the remaining steps is to identify these abstract limits.

	\smallskip
	\noindent\emph{Step 2. Compactness away from \(\Gamma_t\).}
	The estimates \eqref{estimate-1}--\eqref{estimate-2} imply
	\begin{equation}\label{fixed-cylinder-bounds}
		\sup_{t\in (0,T)}\norm{\aae(t)}_{H^1(\o\setminus\Gamma_t(\delta))}
		+\norm{\p_t\aae}_{L^2((0,T);L^2(\o\setminus\Gamma_t(\delta)))}
		\le C_\delta,
	\end{equation}
	where \(C_\delta\) is independent of \(\ee\). The standard Aubin--Lions--Simon compactness criterion, the family \(\{\aae\}_\ee\) is relatively compact in \(C((0,T);L^2(\Omega\setminus\Gamma_t(\delta)))\).  Hence, after passing to a subsequence depending on \(\delta\), there is a map \(\aa_\delta^\pm\) such that
	\begin{subequations}\label{local-cylinder-conv}
		\begin{align}
			\aae       & \to\aa_\delta^\pm
			           &                                              & \text{strongly in }C((0,T);L^2(\Omega\setminus\Gamma_t(\delta))),\label{local-cylinder-strong}       \\
			\p_t\aae   & \rightharpoonup\p_t\aa_\delta^\pm
			           &                                              & \text{weakly in }L^2((0,T);L^2(\Omega\setminus\Gamma_t(\delta))),\label{local-cylinder-time}         \\
			\nabla\aae & \stackrel{*}{\rightharpoonup}\nabla\aa_\delta^\pm
			           &                                              & \text{weakly-* in }L^\infty((0,T);L^2(\Omega\setminus\Gamma_t(\delta))).\label{local-cylinder-space}
		\end{align}
	\end{subequations}
	The limit is unique on overlaps of such cylinders, because the convergence is strong in \(L^2\).  Therefore the local limits patch together to a map on \(\Omega_T^\pm\). Using a diagonal extraction with $\delta=\delta_k=1/k$, we obtain a single subsequence, still denoted by \(\ee_k\), and maps \(\aa_\pm\) on \(\Omega_T^\pm\) such that \eqref{limit-t}, \eqref{limit-i}, and \eqref{limit-strong} hold on every compact subset of \(\Omega_T^\pm\).  Passing to a further subsequence if necessary, the strong local convergence also gives the a.e. convergence \eqref{ae-convergence}.

	\smallskip
	\noindent\emph{Step 3. Identification of the wells.}
	Let \(K\Subset\Omega_T^\pm\).  For all sufficiently small \(\delta\), the compact set \(K\) is contained in \(\bigcup_t\Omega\setminus\Gamma_t(\delta)\times\{t\}\).  From \eqref{estimate-1},
	\[
		\int_K F(\aaek)\,\dx\dt
		\le \ee_k^2
		\int_K \frac{F_\ee(\aaek)}{\ee_k^2}\,\dx\dt
		\le C_K\ee_k^2\to0 .
	\]
	Since \(\aaek\to\aa_\pm\) a.e. on \(K\), Fatou's lemma gives \(F(\aa_\pm)=0\) a.e. on \(K\).  Thus \(\aa_\pm\in\on\) a.e. in \(\Omega_T^\pm\).
    
   It remains to identify the connected component of \(\on\).
   Since \(F(\aa_\pm)=0\) a.e., we have
   \[
    \det\aa_\pm\in\{1,-1\}
    \qquad\text{a.e. in }\Omega_T^\pm .
  \]
  Moreover, \(\aa_\pm\in H^1_{\loc}(\Omega_T^\pm)\), hence
\(\det\aa_\pm\in H^1_{\loc}(\Omega_T^\pm)\).  Since an \(H^1\)-function
taking only the values \(\pm1\) has vanishing weak gradient, \(\det\aa_\pm\)
is constant on each connected component of \(\Omega_T^\pm\).  The initial
strong convergence \eqref{initial-strong-convergence} fixes this constant:
\[
    \det\aa_+=1,
    \qquad
    \det\aa_-=-1 .
\]
Therefore
\[
    \aa_\pm\in\onpm
    \qquad\text{a.e. in }\Omega_T^\pm .
\]

	\smallskip
	\noindent\emph{Step 4. Global identification.}
	  Let \(\Phi\in C_c^\infty(\Omega_T^\pm;\An)\).  We prove the left spatial commutator convergence; the remaining three cases are identical.  Since the support of \(\Phi\) is compactly contained in one phase, the local convergences from Step 2 apply.  We write
	\[
		\begin{aligned}
			 & \int_{\Omega_T^\pm}
			\bigl([\p_i\aaek,\aaek]_L-[\p_i\aa_\pm,\aa_\pm]_L\bigr):\Phi\,\dx\dt \\
			  =&\int_{\Omega_T^\pm}[\p_i\aaek-\p_i\aa_\pm,\aaek]_L:\Phi\,\dx\dt
			+\int_{\Omega_T^\pm}[\p_i\aa_\pm,\aaek-\aa_\pm]_L:\Phi\,\dx\dt .
		\end{aligned}
	\]
	For the first term, expand the bracket:
	\[
		[\p_i\aaek-\p_i\aa_\pm,\aaek]_L:\Phi
		=(\p_i\aaek-\p_i\aa_\pm):\Phi\aaek
		-(\p_i\aaek-\p_i\aa_\pm):\Phi^\top\aaek .
	\]
	The factor \(\Phi\aaek\) converges strongly in \(L^2\), while \(\p_i\aaek\rightharpoonup\p_i\aa_\pm\) weakly in \(L^2\) on the support of \(\Phi\).  Hence the first term tends to zero.  The second term tends to zero by the strong convergence \(\aaek\to\aa_\pm\) in \(L^2\) and the fact that \(\p_i\aa_\pm\in L^2_{\loc}\).  Thus, for \(0\le i\le d\),
	\[
		\cc_i^L=[\p_i\aa_\pm,\aa_\pm]_L
		\qquad\text{a.e. in }\Omega_T^\pm .
	\]
	The same argument gives
	\[
		\cc_i^R=[\aa_\pm,\p_i\aa_\pm]_R
		\qquad\text{a.e. in }\Omega_T^\pm .
	\]

	\smallskip
	\noindent\emph{Step 5. Regularity up to the interface from the commutators.} 
	For \(\aa_\pm\in\on\), differentiating \(\aa_\pm\aa_\pm^\top=\ii\) gives
	\[
		\p_i\aa_\pm\aa_\pm^\top+\aa_\pm\p_i\aa_\pm^\top=0.
	\]
	Consequently,
	\[
		[\p_i\aa_\pm,\aa_\pm]_L
		=2\p_i\aa_\pm\aa_\pm^\top,
		\qquad
		[\aa_\pm,\p_i\aa_\pm]_R
		=2\aa_\pm^\top\p_i\aa_\pm .
	\]
	Since multiplication by an orthogonal matrix preserves the Frobenius norm,
	\[
		\norm{\p_i\aa_\pm}
		=\frac12\norm{[\p_i\aa_\pm,\aa_\pm]_L}
		=\frac12\norm{[\aa_\pm,\p_i\aa_\pm]_R}.
	\]
	The bounds \eqref{limit-comm-amb} on the weak commutator limits therefore imply
	\[
		\p_t\aa_\pm\in L^2((0,T);L^2(\Omega_t^\pm)),
		\qquad
		\nabla\aa_\pm\in L^\infty((0,T);L^2(\Omega_t^\pm)).
	\]
	This gives the regularity stated in Theorem \ref{thm:main} and completes the proof.
\end{proof}

\section{Harmonic-map heat flow and transmission identities}\label{section:harmonic-map-heat-flow}

In this section we only prove the two weak identities stated in Theorem \ref{thm:main}.  Their interpretation as the bulk harmonic-map heat flow and, in the smooth case, as the Neumann-type transmission condition has already been explained in Remark \ref{rem:main-comments}.

Let \(\Phi\in C_c^\infty((0,T)\times\Omega;\An)\).  Testing the left commutator equation \eqref{commu-eq-} against \(\Phi\) and integrating by parts in space gives
\begin{equation}\label{left-comm-weak-eps}
	\int_0^T\int_\Omega
	\left(
	[\p_t\aae,\aae]_L:\Phi
	+\sumod[\p_i\aae,\aae]_L:\p_i\Phi
	\right)\dx\dt=0 .
\end{equation}
There is no boundary contribution because \(\Phi\) is compactly supported in \((0,T)\times\Omega\).  Passing to the subsequential limit in \eqref{left-comm-weak-eps} by Proposition \ref{prop:limit-convergence},
\begin{equation}\label{left-limit-comm-weak}
	\sum_\pm\int_0^T\int_{\Omega_t^\pm}
	\left(
	[\p_t\aa_\pm,\aa_\pm]_L:\Phi
	+\sumod[\p_i\aa_\pm,\aa_\pm]_L:\p_i\Phi
	\right)\dx\dt=0 .
\end{equation}
Since \(\Phi\in\An\) and \(\aa_\pm\aa_\pm^\top=\ii\),
\[
	[\p_i\aa_\pm,\aa_\pm]_L:\Phi
	=2\p_i\aa_\pm\aa_\pm^\top:\Phi,
	\qquad 0\le i\le d.
\]
Dividing \eqref{left-limit-comm-weak} by two gives
\begin{equation}\label{heat-flow-and-neumann-left}
	\sum_{\pm}\int_0^T\int_{\Omega_t^\pm}
	\left(
	\p_t\aa_\pm\aa_\pm^\top:\Phi
	+\sumod \p_i\aa_\pm\aa_\pm^\top:\p_i\Phi
	\right)\dx\dt=0,
\end{equation}
which is \eqref{inter-equality-left}.

Similarly, for \(\Psi\in C_c^\infty((0,T)\times\Omega;\An)\), testing \eqref{commu-eq-} gives
\begin{equation}\label{right-comm-weak-eps}
	\int_0^T\int_\Omega
	\left(
	[\aae,\p_t\aae]_R:\Psi
	+\sumod[\aae,\p_i\aae]_R:\p_i\Psi
	\right)\dx\dt=0 .
\end{equation}
Using Proposition \ref{prop:limit-convergence} and the identity
\[
	[\aa_\pm,\p_i\aa_\pm]_R:\Psi
	=2\aa_\pm^\top\p_i\aa_\pm:\Psi,
	\qquad 0\le i\le d,
\]
we obtain
\begin{equation}\label{heat-flow-and-neumann-right}
	\sum_{\pm}\int_0^T\int_{\Omega_t^\pm}
	\left(
	\aa_\pm^\top\p_t\aa_\pm:\Psi
	+\sumod \aa_\pm^\top\p_i\aa_\pm:\p_i\Psi
	\right)\dx\dt=0,
\end{equation}
which is \eqref{inter-equality-right}.

\section{Minimal-pair condition}\label{section:minimal-pair-condition}
\subsection{Trace identification}

We now identify the traces of \(\aa_\pm\) on the interface.  The argument follows
the strategy of \cite[Section 5]{Liu2025}; we include the details needed for the
matrix-valued potential.

The mollification error is harmless.  Since \(\dwf\) is globally Lipschitz and
\(\phi_\ee\) is supported in a ball of radius \(O(\ee^K)\),
\[
	\norm{\dwfe-\dwf}_{L^\infty(\Mn)}\le C\ee^K .
\]
Using \(\xi=0\) on \(\p\Omega\), we therefore have
\begin{equation}\label{difference}
	\begin{aligned}
		 & \left|\int_\Omega \xi\cdot\nabla\dwf(\aae)\dx
		-\int_\Omega \xi\cdot\nabla\dwfe(\aae)\dx\right| \\
		 =&\left|\int_\Omega (\div\xi)
		\bigl(\dwf(\aae)-\dwfe(\aae)\bigr)\dx\right|
		\le C(\xi)\ee^K .
	\end{aligned}
\end{equation}
Combining \eqref{difference}, \(\wff\le F\), and \eqref{energy-small}, and then
dividing by \(\ee\), gives the scaled estimate
\begin{equation}\label{relative-entropy-1}
	\sup_{t\in[0,T]}
	\int_\Omega
	\left(
	\frac12\norm{\nabla\aae}^2+
	\frac1{\ee^2}\wff(\aae)
	-\frac1\ee(\xi\cdot\nabla)\dwf(\aae)
	\right)\dx
	\le C .
\end{equation}
This is the form used below on normal line segments.

\begin{lemma}[One-dimensional lower bound; cf. {\cite[Lemmas~4.1--4.2]{Liu2025}}]\label{kappa}
	There exists a non-negative continuous function
	\[
		\kappa:\onp\times\onn\times[0,1]\to[0,\infty)
	\]
	with the following properties.
	\begin{enumerate}[label=\textup{(\arabic*)},leftmargin=2.5em]
		\item \(\kappa(\bb_+,\bb_-,0)=0\) if and only if \(\norm{\bb_+-\bb_-}=2\).
		\item There exist constants \(C_0=C_0(n)>0\), \(\tau_0=\tau_0(n)>0\), and \(\delta_0=\delta_0(n)>0\) such that, for every \(\gamma\in H^1([\!-\delta,\delta];\Mn)\) satisfying
		      \[
			      \gamma(-\delta)\in B_{\delta_0}(\onn),
			      \qquad
			      \gamma(\delta)\in B_{\delta_0}(\onp),
		      \]
		      and every \(\tau\in(0,\tau_0)\),
		      \begin{equation}\label{one-energy}
			      \begin{aligned}
				        \int_{-\delta}^{\delta}
				      \left(
				      \frac12\norm{\gamma'}^2
				      +\frac1{\ee^2}\wff(\gamma)
				      -\frac1\ee(\dwf\circ\gamma)'
				      \right)\dd s               
				       \ge
				      \frac{\min\{C_0\tau^2,
					      \kappa(P\gamma(\delta),P\gamma(-\delta),\tau)\}}
				      {\max\{\ee,\delta\}},
			      \end{aligned}
		      \end{equation}
		      where \(P\) is the nearest-point projection onto \(\on\).
	\end{enumerate}
\end{lemma}
\begin{proof}
We follow the one-dimensional argument of
\cite[Section~4, Lemmas~4.1--4.2]{Liu2025}.  We indicate why it applies
to the present matrix-valued quasi-potential \(\wff\).

The zero set of \(\wff\) is
\[
    \on=\onp\cup\onn .
\]
After choosing \(\delta_0>0\) sufficiently small, the nearest-point
projection
\[
    P:\mathcal U_{\delta_0}(\on)\to \on
\]
is smooth, and \(B_{\delta_0}(\onp)\cup B_{\delta_0}(\onn)\subset
\mathcal U_{\delta_0}(\on)\).  By Lemma~\ref{lemma-quasi-distance-function},
the quasi-distance \(\dwf\) satisfies the Modica-type differential inequality
\[
    \norm{D\dwf(\aa)}
    \le
    \sqrt{2\wff(\aa)}
    \qquad\text{for a.e. }\aa\in\Mn .
\]
Thus \(\dwf\) plays the role of the Agmon primitive associated with the
quasi-potential \(\wff\).

Define the corresponding Agmon-type semi-distance by
\[
    d_{\wff}^{*}(\bb_+,\bb_-)
    :=
    \inf_{\gamma}
    \int \norm{\gamma'(s)}
    \sqrt{2\wff(\gamma(s))}\,\dd s ,
\]
where the infimum is taken over all absolutely continuous curves joining
\(\bb_-\) to \(\bb_+\).  The defect function \(\kappa\) in the statement is
defined as in \cite[Lemma~4.1]{Liu2025}, with the quasi-potential
\(\wff\) and the semi-distance \(d_{\wff}^{*}\).  The proof of
\cite[Lemma~4.1]{Liu2025} then gives that \(\kappa\) is continuous,
non-negative, and satisfies
\[
    \kappa(\bb_+,\bb_-,0)=0
    \quad\Longleftrightarrow\quad
    (\bb_+,\bb_-)\ \text{is a minimal pair for the quasi-potential }
    \wff .
\]
For the present Saint Venant--Kirchhoff potential, the minimal pairs are
characterized by Lemma~\ref{minimal-equi}; hence the preceding condition is
equivalent to
\[
    \norm{\bb_+-\bb_-}=2 .
\]

Applying the one-dimensional estimate
\cite[Lemma~4.2]{Liu2025} to the quasi-potential \(\wff\) gives, for every
\(\gamma\in H^1([-\delta,\delta];\Mn)\) satisfying
\[
    \gamma(-\delta)\in B_{\delta_0}(\onn),
    \qquad
    \gamma(\delta)\in B_{\delta_0}(\onp),
\]
and every \(\tau\in(0,\tau_0)\),
\[
    \int_{-\delta}^{\delta}
    \left(
    \frac12\norm{\gamma'}^2
    +\frac1{\ee^2}\wff(\gamma)
    -\frac1\ee(\dwf\circ\gamma)'
    \right)\dd s
    \ge
    \frac{
    \min\{C_0\tau^2,
    \kappa(P\gamma(\delta),P\gamma(-\delta),\tau)\}
    }{\max\{\ee,\delta\}} .
\]
This is exactly \eqref{one-energy}.
\end{proof}

\begin{proof}[Proof of the minimal-pair condition]
	We prove the assertion for every time \(t\) outside a null set for which the trace
	statements below hold.  Suppose, by contradiction, that the minimal-pair condition
	fails at such a time.  Then there are \(\alpha>0\) and a compact set
	\(E_t^*\subset\Gamma_t\) with \(\hn(E_t^*)\ge\alpha\) such that
	\((\aa_+(p,t),\aa_-(p,t))\) is not a minimal pair for \(p\in E_t^*\).  By the
	positive definiteness of \(\kappa\) and Lemma \ref{minimal-equi}, after possibly
	shrinking \(E_t^*\), there exists \(\beta^*>0\) such that
	\begin{equation}\label{kappa-positive}
		\kappa(\aa_+(p,t),\aa_-(p,t),0)\ge \beta^*
		\qquad\text{for }p\in E_t^* .
	\end{equation}

	Let \(\nu\) be the normal pointing from \(\Omega_t^-\) to \(\Omega_t^+\).  The local strong convergence in Proposition \ref{prop:limit-convergence},
	and Fubini's theorem imply that there exists a null set
	\(N\subset(0,\delta_0)\) such that, for every
	\(\delta\in(0,\delta_0)\setminus N\),
	\begin{subequations}\label{trace-convergences}
		\begin{align}
			\aaek(p+\delta\nu(p),t) & \to\aa_+(p+\delta\nu(p),t)
			                        &               & \text{strongly in }L^2(\Gamma_t), \\
			\aaek(p-\delta\nu(p),t) & \to\aa_-(p-\delta\nu(p),t)
			                        &               & \text{strongly in }L^2(\Gamma_t).
		\end{align}
	\end{subequations}
	The trace
	theorem combined with a diagonal argument gives \(\delta_k\downarrow0\), with
	\(\delta_k\notin N\), such that
	\begin{equation}\label{strong-1}
		\aaek(p\pm\delta_k\nu(p),t)\to\aa_\pm(p,t)
		\quad\text{strongly in }L^2(\Gamma_t).
	\end{equation}
	By Egorov's theorem, we can choose a compact subset
	\(E_t\subset E_t^*\) with \(\hn(E_t)>0\) such that the convergence in
	\eqref{strong-1} is uniform on \(E_t\).  Hence, for all sufficiently large \(k\),
	\[
		\aaek(p-\delta_k\nu(p),t)\in B_{\delta_0}(\onn),
		\qquad
		\aaek(p+\delta_k\nu(p),t)\in B_{\delta_0}(\onp),
		\qquad p\in E_t .
	\]
	Since \(\kappa\) is uniformly continuous on the compact set
	\(\onp\times\onn\times[0,1]\), we first choose
	\(\tau\in(0,\tau_0)\) so small that
	\[
		\kappa(\bb_+,\bb_-,\tau)
		\ge \kappa(\bb_+,\bb_-,0)-\frac14\beta^*
		\qquad\text{for all }(\bb_+,\bb_-)\in\onp\times\onn .
	\]
	Using the uniform convergence on \(E_t\) and the continuity of the nearest-point
	projection \(P\), we then get, for all large \(k\),
	\begin{equation}\label{beta}
		\inf_{p\in E_t}
		\kappa\!\left(
		P\aaek(p+\delta_k\nu(p),t),
		P\aaek(p-\delta_k\nu(p),t),
		\tau
		\right)
		\ge \frac12\beta^*=:\beta>0 .
	\end{equation}

	For \(p\in E_t\), define the normal curve
	\[
		\gamma_{p,k}(s):=\aaek(p+s\nu(p),t),
		\qquad -\delta_k\le s\le\delta_k .
	\]
	Applying Lemma \ref{kappa} to \(\gamma_{p,k}\) and using \eqref{beta}, we obtain
	\begin{equation}\label{normal-lower-bound}
		\begin{aligned}
			 & \int_{-\delta_k}^{\delta_k}
			\left(
			\frac12\norm{\p_s\aaek}^2
			+\frac1{\ee_k^2}\wff(\aaek)
			-\frac1{\ee_k}\p_s(\dwf(\aaek))
			\right)\dd s                   \\
			 & \qquad\ge
			\frac{\min\{C_0\tau^2,\beta\}}{\max\{\ee_k,\delta_k\}}
			\qquad\text{for every }p\in E_t .
		\end{aligned}
	\end{equation}
	Integrating over \(E_t\) gives
	\begin{equation}\label{diver-inter}
		\begin{aligned}
			 & \int_{E_t}\int_{-\delta_k}^{\delta_k}
			\left(
			\frac12\norm{\p_s\aaek}^2
			+\frac1{\ee_k^2}\wff(\aaek)
			-\frac1{\ee_k}\p_s(\dwf(\aaek))
			\right)\dd s\,\dd\hn                     \\
			 & \qquad\ge
			\frac{\hn(E_t)\min\{C_0\tau^2,\beta\}}
			{\max\{\ee_k,\delta_k\}} .
		\end{aligned}
	\end{equation}

	We now obtain the contradictory upper bound.  In the tubular coordinates
	\(x=p+s\nu(p)\), the Jacobian is \(J(p,s)=1+\mathcal{O}(\abs{s})\).  Moreover,
	\(\xi=\nu+\mathcal{O}(s^2)\) near \(\Gamma_t\), and the coercivity estimate gives
	\[
		\frac1{\ee_k}\int_\Omega \abs{\dgt}^2\abs{\nabla\dwf(\aaek)}\dx\le C .
	\]
	Therefore the difference between \(\ee_k^{-1}\xi\cdot\nabla\dwf(\aaek)\) and
	\(\ee_k^{-1}\p_s\dwf(\aaek)\) on the normal cylinder over \(E_t\) is uniformly
	bounded.  Since the tangential-gradient part of \(\norm{\nabla\aaek}^2\) is
	non-negative, the scaled modulated-energy estimate \eqref{relative-entropy-1} and
	the area formula yield
	\begin{equation}\label{bound-inter}
		\begin{aligned}
			 & \int_{E_t}\int_{-\delta_k}^{\delta_k}
			\left(
			\frac12\norm{\p_s\aaek}^2
			+\frac1{\ee_k^2}\wff(\aaek)
			-\frac1{\ee_k}\p_s(\dwf(\aaek))
			\right)\dd s\,\dd\hn
			\le C .
		\end{aligned}
	\end{equation}
	The right-hand side of \eqref{diver-inter} tends to \(+\infty\) because
	\(\max\{\ee_k,\delta_k\}\to0\), whereas \eqref{bound-inter} is uniformly bounded.
	This contradiction proves \eqref{result-minimal}.
\end{proof}

\section*{Acknowledgments}
The author thanks Fanghua Lin and Yuning Liu for helpful discussions, and thanks
his supervisor Yaguang Wang for his guidance and support. This research was supported by the National Natural Science Foundation of China under Grant No. 12331008. 

\vspace{0.5em}
\paragraph*{Data availability statement}
No data were generated or analyzed in this work.

\vspace{0.5em}
\paragraph*{Conflict of interest statement}
The author declares no competing interests.

\end{document}